\documentclass[11pt]{article}
\usepackage{amsmath,amssymb,amsthm,color,graphicx, fancybox}

\usepackage{subfigure}
\usepackage{subfigmat}
\usepackage{float}

\usepackage{tcolorbox}
\usepackage{bclogo}
\usepackage[latin1]{inputenc}
\usepackage{mdframed}
\usepackage{amsmath}
\usepackage{amsfonts}
\usepackage{amssymb}
\usepackage{wasysym}
\usepackage{amsthm}
\usepackage{graphicx}

\usepackage[top=2.5cm, bottom=2.2cm, left=2.8cm, right=3cm]{geometry}
\newcommand{\be}{\begin{equation}}

\newcommand{\ee}{\end{equation}}
\newcommand{\bd}{\begin{displaymath}}
\newcommand{\ed}{\end{displaymath}}
\newcommand{\eps}{\varepsilon}

\newcommand{\bb}{\mathbf{b}}

\def\bbc#1#2{{\rm \mkern#2mu\vbar\mkern-#2mu#1}}
\def\bbb#1{{\rm     I\mkern-3.5mu      #1}}      \def\bba#1#2{{\rm
		#1\mkern-#2mu\fudge #1}}
\def\bb#1{{\count4=`#1 \advance\count4by-64  \ifcase\count4\or\bba
		A{11.5}\or \bbb B\or\bbc C{5}\or\bbb D\or\bbb E\or\bbb F \or\bbc
		G{5}\or\bbb H\or \bbb I\or\bbc J{3}\or\bbb K\or\bbb  L  \or\bbb
		M\or\bbb N\or\bbc O{5} \or \bbb P\or\bbc Q{5}\orrrr\b bb
		R\or\bbc S{4.2}\or\bba T{10.5}\or\bbc   U{5}\or      \bba
		V{12}\or\bba W{16.5}\or\bba X{11}\or\bba Y{11.7}\or\bba
		Z{7.5}\fi}}
\def\rat{\hbox{{\rm Q}\kern-.70em\hbox{{\rm I} } }}
\def \R {\bbb R}

\newtheorem{theorem}{Theorem}

\newtheorem{lemma}{Lemma}[section]

\newtheorem{proposition}{Proposition}[section]

\newtheorem{problem}{Problem}

\newtheorem{acknowledgment*}{Acknowledgment}
\newtheorem{remark}{Remark}

\begin{document}
	\title{Is the mailing Gilbert-Steiner problem  convex?}
	\author{Gershon Wolansky \\ Department of Mathematics, Technion, Haifa 32000, Israel}

\maketitle
\begin{center} Abstract \end{center}	A convexification of the mailing version of the  finite Gilbert problem for optimal networks is introduced.  It is ia convex functional  on the set of probability measures subject to the Wasserstein $p-$ metric. 
The minimizer of this convex functional is a measure supported in a graph. If this graph is a tree (i.e contains no cycles)  then this tree is also a minimum of the corresponding mailing Gilbert problem.  The convexification of the Steiner problem is the limit of these convexified Gilbert's problems in the limit $p\rightarrow\infty$. 
A numerical algorithm for the implementation of the convexified Gilbert-mailing problem is also suggested, based on entropic regularization. 

	\section{Introduction}
	Optimal branched transportation is a variant of the theory of Monge-Kantorovich \cite{[K], vi} on optimal transportation. 
	  The classical Kantorovich cost of transporting a probability measure $f^+(dx)$ to another probability measure $f^-(dy)$, where the cost of transporting one unit from $x$ to $y$ is $c(x,y)$,  is given by
	  \be\label{MKan} \min_{\pi\in\Pi(f^+, f^-)} \int\int c(x,y)\pi(dxdy)\ee
	  where
	 $\Pi(f^+, f^-)$ stands for the set of all {\em transport plans} $\pi$, which are 2-points probability measures whose marginals are $f^\pm$ \cite{[V], vi}. 
	   In contrast to the classical Monge-Kantorovich theory where each of the "mass particles" is transported independently of the others,   in the  branched transport (as well as a congestion transport)    the particles  are assumed to interact  with each other while moving from a source to a target distribution.
	Thus, while Monge-Kantorovich optimal transport (\ref{MKan}) is, basically, a linear programming in the affine space of 
	transport plans  given by probability measures (even though it is highly nonlinear in the domain of {\em deterministic} transport plans), a branched (and congested \cite{BJO, Cari}) transport is a minimizer of nonlinear functional even on the affine space of probability measures. Moreover, while the functional corresponding to congestion transport is, in general, a convex one, the branched optimal transport theory  involves a minimization of  non-convex functionals. In particular the optimal solution is not unique, in general, and the computation task of optimal transport is challenging  \cite{OS, YJCL}.  
	\par
	The cost function of a branched transport  induces  ramifications in a natural way.  It simulates many natural phenomena  such as roots systems of trees, 
	leaf ribs and the nervous system, as well as  supply-demand distribution networks such as irrigation networks and 
	electric power supply. The common principle behind all these cases is that the 
	the cost functions privilege large flows over diffusive ones, which causes the orbits of the transport  to concentrate on 1-dimensional currents. 
\par
One of  first models for   branched  transport was introduced by
Gilbert \cite{Gil} (see also \cite{zang}). Given source  ${{{f^+}}}$ and target $f^-$ probability measures on $\R^d$  supported each  in  a {\em finite set   of points},  
he   considered optimization over
finite directed  tree $T=(E,V)$ whose vertices $V$ contain  $supp({{{f^+}}})\cup supp(f^-)$.    Gilbert  suggested to minimize 
\be\label{Gil} \sum_{e\in E(T)}w^\sigma(e)|e|\ee
over {\em all} such trees  $T$ and weight functions  $w:E(T)\rightarrow \R^+$ representing the fluxes at  the edges. Here $|e|$ is the length of an edge $e$ and $w$  satisfies
\be\label{const} \sum_{e\in E^\pm(v)} w(e)=f^\pm(v) \ , v\in supp(f^\pm) \ ; \ \ \sum_{e\in E^+(v)}w(e)= \sum_{e\in E^-(v)}w(e) \ \ v\in V-(supp({{{f^+}}})\cup supp(f^-)) \ee
 where $E^+(v)$ ($E^-(v)$) stands for all edges outgoing  from (ingoing to)  $v$. 
  The exponent $\sigma$ is chosen to be a number in the interval $[0,1)$, reflecting the preference for concentration of the flow, due to the inequality $w_1^\sigma(e)+w_2^\sigma(e)\geq \left(w_1(e)+w_2(e)\right)^\sigma$. This choice of $\sigma$ is the source of non-convexity (and non-uniqueness) of this problem. In particular, the case $\sigma=0$ corresponds to the  Steiner problem of minimal graphs.


 In 2003  Xia \cite{Xia1} (see also \cite{Xia2}  and references therein) 
extended this model to a continuous framework using Radon vector-valued measures $u$ in on $M$ which satisfy the  condition $\nabla\cdot u= {{{f^+}}}-f^-$ in sense of distributions. This is a weak formulation of (\ref{const}). The optimal transport network $u$ is obtained by minimizing
$$ {\cal M}^\sigma(u):= \int_M \theta^\sigma(u) d{\cal H}^1 $$ 
over all $u$ satisfying the above constraint, 
where  $\theta(u)_{(x)}$ is the corresponding occupation density of $u$ at $x$ (which is a weak formulation of (\ref{Gil})). Such an extension makes sense for general Borel probability measures, and is reduced to Gilbert's formulation (\ref{Gil}, \ref{const}) if $f^\pm$ are supported on a finite set of points. One of the fundamental results of Xia is the condition 
 $\sigma> 1-1/d$ which guarantees a finite value of the transport cost  for {\em any} Borel measures $f^\pm$  in $\R^d$ representing the source and sink distributions (see also \cite{CDRM, CDRM1}). 

Another approach
 \cite{BCM, MSM} represents  traffic plans as measures on the set of Lipschitz
paths connecting the source and sinks. The functional then acts on probability measures  on this set.   This approach is equivalent (in Euler representation) to Xia formulation (see \cite{BCM, CDRM}), and it can also accommodate  the mailing problem. 

In another approach,  (see \cite {BSp} and references therein), a transportation network is modeled as a connected set $\Sigma\subset\R^n$, and the linear Monge-Kantorovich problem is introduced to the metric $d_\Sigma(x,y)=d(x,y)\wedge(dist(x, \Sigma)+dist(y, \Sigma))$ as follows:  Given source and target distributions $f^{\pm}$, the problem is reduced to minimizing the Kantorovich cost
$$ \Sigma \Rightarrow\min_\pi\int \int d_\Sigma(x, y)\pi(dxdy)$$
among all transport plans $\pi$ whose marginals are prescribed by $f^{\pm}$, and all connected sets $\Sigma$ whose length ${\cal H}_1(\Sigma)$ are 
bounded by a prescribed  value $L>0$.  The non linearity (and non-convexity) is conveyed in the dependence on $\Sigma$.  This non linearity  persists also in the mailing version, i.e. for minimizing $\int\int d_\Sigma(x,y)\pi(dxdy)$ with respect to $\Sigma$, where a transport  plan $\pi$ is prescribed (rather than its marginals $f^{\pm}$). 

Another
approach (see e.g.\cite{BBS}) extend Benamu Brenier \cite{[BB]} kinetic formulation of optimal transport to this setting. This formulation turns out to be equivalent to Xia's approach.   It seems, however, that the mailing problem cannot be accommodated in this setting. 

Still another approach using a limit theorem was introduced   the author  in \cite{W4}.  

\noindent
All in all, these approaches share the non-convex structure and, as a result, do not guarantee  a  unique solution, in general. 
\section{Formulation of the mailing problem}
In the {\em mailing problem} (or "who goes where" situation \cite{Xia2, BCM, BCM1, CDRM1}),  a transport  plan $\pi$ is prescribed, determining the "mailing address" in the support of $f^{-}$ for  each point in the support of $f^{+}$.

Let us consider first the discrete version of the Gilbert mailing problem. Suppose 
 $A= supp(f^+), B=supp(f^-)$ be finite sets in $\R^d$. {\em The mailing program} from $f^+$ to $f^-$ is a non-negative function $\pi=\pi(x,y)$ on $A\times B$ which satisfies
 \be\label{balgpm} \sum_{x\in A} \pi(x,y)=f^-(y), \ \ \  \sum_{y\in B} \pi(x,y)=f^+(x)\ee
 
 We view $\pi(x,y)$ as the flux of mass from $x\in A$ to $y\in B$.

	A network $T$ in the class $(A,B, \pi)$ is a tree embedded in $\R^d$. This tree is a union of a finite number of segments $e$, called edges.  
	
	If $x\in A$, $y\in B$ and $\pi(x,y)>0$ then there exists a connected component $o_T(x,y)\subset T$ such that $x, y\in o_T(x,y)$.  We refer to $o_T(x,y)$ as the {\em orbit} from $x$ to $y$ along the tree $T$. In particular we may view $o_T(x,y)$ as a subset of the edges $e$ composing $T$.  By definition of a tree, there exists {\em at most one} connected orbit connecting each $x,y\in T$.  In particular, there exists a {\em unique} orbit connecting $x\in A$ to $y\in B$ provided $\pi(x,y)>0$. 

For any edge $e\in T$ let
	\be\label{mudef} w_\pi(e):=\sum_{x\in A, y\in B;  o_T(x,y)\ni e} \pi(x,y) \ . \ee
The {\em length} of an edge $e\in T$ is denoted by $|e|$. 
\begin{remark}
	$w_\pi$ satisfies the Kirchhoff's condition (\ref{const}) on the tree $T$. 
\end{remark}
\begin{problem}\label{gilnocon}
	Let $\sigma\in [0, 1)$. 
	The Gilbert mailing problem associated with $(A,B,\pi)$ is:
	
	Minimize 
	$$G(T;\pi):= \sum_{e\in T} w_\pi(e)^\sigma  |e|$$ 
	among all trees $T$ in the class $(A,B,\pi)$ 
\end{problem}
\begin{proposition}
	The discrete Gilbert's problem (\ref{Gil}) is obtained by minimizing $G(\pi):= \min_TG(T,\pi)$ over all plans $\pi$ satisfying (\ref{balgpm}). 
\end{proposition}
We now introduce an equivalent formulation of the Gilbert's mailing problem: 

Let  $s(e)\geq 0$ represents  the cost of construction the edge $e$ {\em per unit length}. Thus, the cost of construction of $e$ is just $|e|s(e)$. We impose the limit budget of constructing the network $T$ by
\be\label{consG}\sum_{e\in T} |e|s(e)\leq 1 \ . \ee
The cost of transporting a unit of mass along the edge  depends on the cost of construction of this edge in a monotone decreasing way: the higher the investment in the segment, the easier (and cheaper) is the cost of the transport trough this edge. We presume  that the cost of transporting a unit  mass per unit length in an edge $e$ is $s^{-\alpha}$ where $\alpha >0$.  Thus, the cost of transporting a unit of mass from point $x\in A$ to $y\in B$  trough the network $T$ is $\sum_{e\in o_T(x,y)}|e| s(e)^{-\alpha}$.

 We define the network for the mailing problem as follows: 
\begin{problem}\label{congil}
	Minimize 
$$  H(T, s):= \sum_{x\in A}\sum_{y\in B} \pi(x,y)\sum_{e\in o_T(x,y)} |e|s(e)^{-\alpha}$$
over all trees $T$ in class $(A,B,\pi)$,   and functions $s=s(e)$ on $T$ satisfying the constraint (\ref{consG}). 
\end{problem}
\begin{lemma}
	The mailing Gilbert's problem \ref{gilnocon}  and Problem \ref{congil} are equivalent, provided $\sigma=\frac{1}{\alpha+1}$. 
\end{lemma}
\begin{proof}
	By (\ref{mudef}) we rewrite $H(T,s)$ as
	$$ H(T,s)=\sum_{e\in T} w_\pi(e) |e|s^{-\alpha}(e)$$ Let 
	$$H(T)=\min_s H(T,s)    $$
	where the minimum is subjected to the constraint (\ref{consG}). 
For a given tree $T$ this is a strictly convex function of $s$. There is a unique global minimizer under the constraint (\ref{consG}). Since $\alpha >0$ this minimizer is obtained at positive $s$ for any edge $e$. Let $\lambda$ be the Lagrange multiplied due to the constraint (\ref{consG}). Then
$$-\alpha w_\pi(e) s^{-\alpha-1}(e)+ \lambda=0$$
for any $e\in T$, thus $s(e)=(\lambda/\alpha)^{-1/(\alpha+1)}
w_\pi(\alpha)^{1/1+\alpha}$.   Substituting this in the constraint (which must holds with equality) implies
$$ \left(\frac{\lambda}{\alpha}\right)^{\alpha/ (\alpha+1)}
=\left(\sum_{e\in T}|e|w_\pi(e)^{1/(\alpha+1)}\right)^\alpha  \ . $$

 Substitute $s(e)$  in $H$ we get $H(T)=(\alpha/\lambda)^{\alpha/ (\alpha+1)}\sum_{e\in T}w_\pi(e)^{1/(1+\alpha)}|e|$. Thus
 $$ H(T)=\left(\sum_{e\in T}|e|w_\pi(e)^{1/(\alpha+1)}\right)^{1+\alpha} \  $$
 which is the $1+\alpha$ power of the Gilbert-mailing cost $G(T,\pi)$ under $\sigma=1/(\alpha+1)$. 
\end{proof}
\section{Conditional Wasserstein metric}
We consider now $\nu_\pm$ as a  pair of  Borel probability measures on $\R^d$. 
Recall the Wasserstein metric  on the set of probability Borel measures ${\cal P}(\R^d)$
 $$W_p(\nu_+,\nu_-)=\left( \min_{\pi\in\Pi(nu_+,\nu_-)} \int\int |x-y|^p\pi(dxdy)\right)^{1/p}$$
 where $\Pi(\nu_+, \nu_-)$ is the continuum version of (\ref{balgpm})
\be\label{Pi} \Pi(\nu_+,\nu_-)= \left\{ \pi \in {\cal P}(\R^d\times \R^d); \int_{y\in \R^d} \pi(dx,dy)=\nu_+(dx), \ \ \  \int_{x\in \R^d} \pi(dx,dy)=\nu_-(dy)\right\}\ee
 It is known that $W_p$ is a metric on the set of Borel probability measures on $\R^d$ having a finite $p-$ moment, if $p\geq 1$ (see, e.g.\cite{[GM]}). Let us define the {\em conditional Wasserstein}  $p-$ metric  $W_p(\nu_+, \nu_-\|\mu)$, where $\mu\in {\cal P}(\R^d)$:
 	$$ W_p(\nu_+, \nu_-\|\mu):= \lim_{\eps\rightarrow 0}
 	\eps^{-1}W_p(\mu+\eps \nu_+; \mu+\eps \nu_-) \ . $$
	\begin{theorem}\cite{W2, W3}\label{firstT}
	$$	W_p(\nu_+, \nu_-\|\mu)= \sup\left\{\int \phi d (\nu_+-\nu_-) , \ \ \phi\in C^1(\R^d), \ \ \ , \ \int
 			|\nabla\phi|^{p^{'}} d\mu\leq 1\right\} \  $$ 
 			where $p^{'}= \frac{p}{p-1}$. In particular, 
 			$$ 	\mu\rightarrow  W^p_p(\nu_+, \nu_-\|\mu)\equiv p\sup_{\phi\in C^1}\int \phi d (\nu_+-\nu_-)-\frac{p}{p^{'}}\int |\nabla\phi|^{p^{'}}d\mu $$
			is a convex function in $\mu\in {\cal P}$ for fixed $\nu_\pm$. 
 	\end{theorem}
 Let us now substitute $\delta_x$ for $\nu_+$ and $\delta_y$ for $\nu_-$.
 \be\label{Wpineq} W_p(x,y\|\mu):= W_p(\delta_x, \delta_y\|\mu)\equiv \sup_{\phi\in C^1, \phi\not\equiv 0}\frac{\phi(x)-\phi(y)}{\left(\int |\nabla\phi|^{p^{'}}d\mu\right)^{1/p^{'}} }\ . \ee
 It follows that 
 \begin{lemma}
$\mu\rightarrow W_p^p(x,y\|\mu)\in R_+\cup\{\infty\}$
 is a convex function in $\mu\in {\cal P}$ for any $x,y\in \R^d$. In addition, $W_p(\cdot, \cdot\|\mu)$ is a an extended metric on $\R^d$ (i.e. it may attains infinite values), for fixed $\mu\in{\cal P}$. 
 \end{lemma}
\section{Relaxation of the Gilbert-mailing problem}
Let now $T$ be a tree in class $(A,B,s)$. We associated with this tree a probability  measure $\mu_T$ which is supported on this tree and such that $\mu_T(e)=|e|s(e)$ for any $e\in T$. 
\begin{lemma}\label{lemma1}
	If $x\in A, y\in B$ then $W_p^p(x,y\|\mu_T)\geq \sum_{e\in o_T(x,y)}|e|s^{1-p}(e)$. The equality holds iff $\mu_T$ is a uniform measure on each edge $e$.
	\end{lemma}
\begin{proof}
Let $\cup_{i=1}^l e_i=o_T(x,y)$, $|o_T(x,y)|:=\sum_{e\in o_T(x,y)}|e|$. 
 $q:[0,|o_T(x,y)|]\rightarrow 0_T(x,y)$ be an arc-length  parameterization of the orbit, and $\rho$ a density of a positive measure on $[0,|o_T(x,y)|]$ such that $q_\#(\rho d\tau) = \mu_T\lfloor_{o_T(x,y)}$. 
that is, for any test function $\phi\in C_b(\R^d)$:
$$ \int_{o_T(x,y)} \phi d\mu_T=\int_0^{|o^T(x,y)|}\phi(q(\tau))\rho(\tau) d\tau\ .  $$
Let 
$\tau_i=q^{-1}(e_i\cap e_{i+1})$ for $i=1\ldots l-1$. We set $\tau_0=q^{-1}(x)=0$ and $\tau_l=q^{-1}(y)=|o_T(x,y)|$. Since $q$ is an arc-length parameterization and  
$\mu_T(e_i)= |e_i|s(e_i)$,  we get \be\label{consrho}\int_{\tau_i}^{\tau_{i+1}}\rho d\tau= |e_i|s(e_i) \ \text{and} \ \ \tau_{i+1}-\tau_i=|e_i| \ .   \ee

 Let $\phi\in C_b^1(\R^d)$ and  $\psi(\tau)=\phi(q(\tau))$.  Then
$$ \int_{o_T(x,y)}|\nabla\phi|^{p^{'}}d\mu_T \geq \int_0^{|o_T(x,y)|}|\psi^{'}|^{p^{'}}\rho(\tau)d\tau \ , \ \ \ \phi(y)-\phi(x)=\psi(|o_T(x,y)|)-\psi(0) \ . $$
A direct calculation yields
$$ \int_0^{|o_T(x,y)|}\rho|\psi^{'}|^{p^{'}}d\tau \geq |\psi_i(|o_T(x,y)|)-\psi_i(0)|^{p^{'}}\left(\int_0^{|o_T(x,y)|}\rho^{1-p}\right)^{1/(1-p)} \  , $$
hence (since $p=p^{'}/(p^{'}-1)$)
$$\left(\frac{\psi({|o_T(x,y)|})-\psi(0)}{\int_0^{|o_T(x,y)|} \rho|\psi^{'}|^{p^{'}}}\right)^p\leq \int_0^{|o_T(x,y)|}\rho^{1-p} $$
holds for any smooth function $\psi$ on $[0,1]$. Moreover, we can find a sequence of such smooth functions $\psi_n$ for which the limit above holds with an equality, or the left hand side blows to infinity, if the integral on the right hand side diverges.  Each such smooth $\psi$ can be lifted to a Lipschitz function on  the orbit $o_T(x,y)$ via 
$\psi\circ q^{-1}$, and we can extend $\psi\circ q^{-1}$ to  a function $\phi$ the entire tree $T$ such that $\phi$ is a constant $=\psi(\tau_i)$ on each sub-tree rooted at $e_{i+1}\cap e_i$. Since all these sub-tree are disjoint, we obtain from $\psi$  a Lipschitz function  $\phi$ which can be extended to $\R^d$ such that
$$\left(\frac{\phi(x)-\phi(y)}{\int_{\R^d}|\nabla\phi|^{p^{'}}d\mu}\right)^p\leq \int_0^{|o_T(x,y)|}\rho^{1-p} \ .  $$
Since the supremum on the left yields the equality it follows by (\ref{Wpineq}) that $W_p^p(x,y)\|\mu_T)= \int_0^{|o_T(x,y)|}\rho^{1-p}$. 

By (\ref{consrho}) and the Jensen's inequality  it follows that
$$ \int_0^{|o_T(x,y)|}\rho^{1-p} \geq \sum_{e\in o_T(x,y) }|e|s^{1-p}(e) $$ and the equality is attained at $\rho(\tau)= s(e)$ on $\tau\in (\tau_i, \tau_{i+1})$ (i.e $\mu_T$ is a the uniform measure $s(e)$ on each edge $e\in o_T(x,y)$). 

\end{proof}
The convexified Gilbert-mailing problem for a given $\sigma\in(0,1)$ takes the form:
\begin{tcolorbox}
Minimize $H_p(\mu):= \sum_{x\in A}\sum_{y\in B} \pi(x,y) W_p^p(x,y\|\mu)$, over $\mu\in {\cal P}$, where $p=1/\sigma$. 
	\end{tcolorbox}
Since $H_p$ is a convex functional on ${\cal P}$ we obtained 
\begin{tcolorbox}
\begin{theorem}
	There exists a unique solution to the Gilbert-mailing problem, supported in a finite graph. If the support is a tree then this tree is a solution of the mailing Gilbert problem. 
\end{theorem}
\end{tcolorbox}
\begin{remark}
The possibility of closed cycles cannot be excluded. Suppose there are two orbits $o_1(x,y),\ldots o_k(x,y)$ connecting $x$ and $y$. Let $\mu$ be a  positive measure supported on $\cup o_i(x,y)$, $\rho_i$ being the density on $o_i(x,y)$. By a computation similar to Lemma \ref{lemma1} we 
obtain
$$ W_p^p(x,y\|\mu)=\left( \sum_{i=1}^k \left(\int_{o_i(x,y)} \rho_i^{1-p}\right)^{1/(1-p)}\right)^{1-p} \ . $$
\end{remark}
\begin{remark}
	$\sigma=0$ (the Steiner mailing case) corresponds to the limit $p=\infty$. In particular, the mailing Steiner problem  is a limit of convex optimization problem. Note that the mailing Steiner problem is equivalent to the Steiner problem itself if $\pi(x,y)>0$ for any $x\in A, y\in B$.  \end{remark}
\section{Numerical implementation and entropic relaxation of the mailing problem}
	Let us consider a finite grid $Z\subset \R^d$. Let $A:=\{x_1, \ldots x_m\}$, $B=\{y_1, \ldots y_n\}$ subsets of  $Z$.

On the grid $Z$ we assign weights $m$ in the simples  $$S(Z):= \left\{m:Z\rightarrow\R_+; \ \  m(z)\geq 0, \ \ \sum_{z\in Z} m(z)=1\right\} \ .  $$
 This is a discretization of the  probability measures $\mu$ on $Z$.

For $x_i, y_j\in A,B$,  the discrete version of $W_p^p(x_i,y_j\|\mu)$ is the supremum over $\phi\in \R^Z$ of 
$$W^p(m,\phi; x_i, y_j)=- \frac{p}{p^{'}}\sum_{z\in Z}\sum_{z^{'}\in N(z)} m(z) \frac{1}{|N(z)|}|\phi(z)-\phi(z^{'})|^{p^{'}} 
+ p\left(\phi(x_i) -\phi(y_j)\right) \  $$
(c.f. Theorem {\ref{firstT}).
Here $N(z)$ stands for the neighbors in $Z$ of the grid point $z$ and $|N(z)|$ the cardinality  of $N(z)$ (recall $x_i, y_j\in Z$ as well). 
Our object is to find
\be\label{reg} \min_{m\in S(Z)}\sum_{x_i\in A, y_j\in B} \pi(i,j)\max_{\phi\in \R^Z} W^p(m,\phi; x_i, y_j)\ee
where $\pi(i,j):= \pi(x_i, y_j)$ is the transport plan from $A$ to $B$. The support of the minimizer $m$ in $S(Z)$ is the discrete approximation of  the optimal tree solving the mailing Gilbert problem for $\sigma=1/p$. 

Let 
$$ \phi_{i,j}(m):= \arg\max W^p(m,\phi; x_i, y_j) \ . $$
The entropic regularization of (\ref{reg}) is
\be\label{regeps} \min_{m\in S(Z)}\sum_{x_i\in A}\sum_{y_j\in B} \pi(i,j)
W^p(m, \phi_{i,j}(m); x_i, y_j) +(\eps p/p^{'})\sum_{z\in Z} m(z)\ln (m(z)) \ . \ee

The minimal $m\in S(Z)$ is 
\be\label{mdef}m(z)= \frac{\exp\left(\sum_{i,j}\pi(i,j)D \phi_{i,j}(z)/2\eps\right)}{\sum_{z^{'}\in Z} \exp\left(\sum_{i,j}D \phi_{i,j}(z^{'})/2\eps\right)}\ee
where 
\begin{equation}\label{DeltaU} D \phi(z):= \sum_{z^{'}\in N(z)}  \frac{1}{|N(z)|}|\phi(z)-\phi(z^{'})|^{p^{'}}
\end{equation}
Substitute this in (\ref{regeps}) and using MinMax Theorem  we obtain the equivanent problem: Maximize over $\{\phi\}\in \R^{Z\times A\times B}$:
$$  H^\eps(\{\phi\})=-\eps\ln\left(\sum_{z\in Z}\exp\left(\sum_{i,j}\pi(i,j)D\phi_{i,j}(z)/2\eps
\right)
\right)+ p\sum_{i,j} \pi(i,j)\left(\phi_{i,j}(x_i) -\phi_{i,j}(y_j)\right) \ . 
$$
where $D \phi$ defined by (\ref{DeltaU}).

\end{document}